\documentclass{article}

\author{ M. Ciavarella, L. Terracini}

\usepackage{amssymb}

\newcommand\qed{\hfill \rule{6pt}{6pt}}
\newenvironment{proof}{\noindent{\bf Proof.\ }\hspace{1pt}}{\hspace{-5pt}\qed\par\bigskip}
\def\gp{{\mathfrak{p}}}
\newcommand{\s}{\times}
\newcommand{\QQ}{{\bf Q}}
\newcommand{\CC}{{\bf C}}
\newcommand{\A}{{\bf A}}
\newcommand{\RR}{{\bf R}}
\newcommand{\FF}{{\bf F}}
\newcommand{\XX}{{\bf X}}

\newcommand{\SL}{{\bf SL}}

\newcommand{\MM}{\mathcal M}

\newcommand{\OO}{\mathcal O}

\newcommand{\mR}{{\mathcal R}}
\newcommand{\ws}{\widehat\psi}
\newcommand{\ts}{\widetilde\psi}
\newcommand{\mm}{\mathfrak m}

\newcommand{\TT}{{\bf T}}

\newcommand{\tp}{{\widetilde \psi}}

\newcommand{\orho}{\overline\rho}

\newcommand{\ZZ}{{\bf Z}}

\newtheorem{theorem}{Theorem}[section]

\newtheorem{Propo1}{Proposition}[section]
\newtheorem{definition}{Definition}[section]
\newtheorem{corollary}{Corollary}[section]
\newtheorem{Lem1}{Lemma}[section]
\newtheorem{conjecture}{Conjecture}[section]
\def\modulo{{\rm mod}}
\def\End{{\rm End}}
\def\Hom{{\rm Hom}}
\def\Frob{{\rm Frob}}
\def\det{{\rm det}}
\def\galois{{\rm Gal}}

\def\tr{{\rm tr}}
\def\det{{\rm det}}
\def\ker{{\rm Ker}}

\def\Res{{\rm Res}}
\def\Cor{{\rm Cor}}

\begin{document}
\title{About an analogue of Ihara's lemma for Shimura curves}
\date{ }
\maketitle

\section*{Abstract}

The object  of this work is to present the status of art of an open problem: to provide an analogue for the cohomology of Shimura curves of the Ihara's lemma \cite{Ihara73} which holds for modular curves. We will formulate our conjecture and locate it in the more general setting of the congruence subgroup problem.
We will exploit the relationship between cohomology of Shimura curves and certain spaces of modular forms to establish some consequences of the conjecture  about congruence modules of  modular forms and about the problem of raising the level.

\bigskip

\noindent{\bf AMS Classification}: 11F80, 11F33, 11G18, 20G30.

\bigskip

\noindent{\bf Keywords}: cohomology of Shimura curves, congruence subgroups, Hecke algebras.
\section*{Introduction}
%The object of this work, is to give a survey of the actual situation of a still open problem: to provide an analogue  for Shimura curves of an Ihara's lemma  which holds for modular curves \cite{Ihara73}.
\noindent Let $N$ be a positive integer and $q$ be a prime number not dividing $N$. We denote by $\mathcal M(N)$ the Hecke modules arising from the cohomology with integer coefficients of the modular curve of level $N$. Ihara's lemma \cite{Ihara73} establishes that the sum of the two natural degeneracy maps $\alpha:\mathcal M(N)^2\to\mathcal M(Nq)$ has torsion-free cokernel. This result has been widely applied to detect congruences between modular forms \cite{Ribet83} and for establishing the freeness of a local component of $\mathcal M(N)$ over the Hecke algebra in the non-minimal case \cite{Wiles95}, \cite{WilesTaylor95}, \cite{DarmonDiamondTaylor95}, \cite{DiamondRibet97}. \\
Ihara's Lemma has been generalized by Diamond and Taylor \cite{DiamondTaylor94} to certain Hecke components of the $\ell$-adic cohomology of Shimura curves  with good reduction at $\ell$.
Their proof relies on a crystalline cohomology argument. The reader can refer to Section \ref{literature} for a more detailed exposition of known results.\\
In this  work (Section 1) we state a conjecture concerning a problem of Ihara for the $\ell$-adic cohomology of Shimura curves in some case of bad reduction. In Section 2 we give a review of the literature about Ihara's problems in some particular situations.  In Section 3 we relate our conjecture to more general conjectures about the subgroup congruence property. In the last section, we use the conjecture to generalize to non-minimal levels some Galois representations theoretic results about local components of Hecke algebras and cohomology of Shimura curves proved in the minimal case in \cite{Terracini03, miri2006}.

\section{The conjecture}\label{shimura}

\noindent Let $B$ be an indefinite quaternion algebra over $\QQ$ of discriminant $\Delta\not =1$ and let $\nu$ be its reduced norm. Let $R$ be a maximal order in $B$.   If $p$ is a prime not dividing $\Delta$, (included $p=\infty$) we fix an isomorphism $i_p:B_p\to M_2(\QQ_p)$ such that  $i_p(R_p)=M_2(\ZZ_p)$ if $p\not=\infty$, where $B_p=B\otimes_{\QQ}\QQ_p$ and $R_p=R\otimes_{\ZZ}\ZZ_p$. \par
We will denote by $B_\A$ the adelization of $B$, by $B_\A^\times$ the topological group of invertible elements in $B_\A$ and $B_\A^{\times,\infty}$ the subgroup of finite id\`{e}les.\\

%For each $U$ open compact subgroup of $B^{\s,\infty}$ let $\XX(U)$ be denotes the corresponding Shimura curve. Thus $\XX(U)$ is a smooth proper curve over $\QQ$, for which the set of complex points is isomorphic to $B^\s\backslash B_\A^\s/UK_\infty$ where $K_\infty$ is the stabilizer of $i\in\CC$ under the usual action of $(B\otimes\A)^\s\cong GL_2(\RR)$ on $\CC-\RR$.
\noindent We fix a prime $\ell>2$ and we suppose $\Delta=\Delta'\ell$.
Let $N$ be an integer prime to $\Delta$; if $p\not|\Delta$, we define:
$$K_p^0(N)=i_p^{-1}\left\{\gamma\in GL_2(\ZZ_p)\ |\ \gamma\equiv \left(\begin {array} {cc}
* & *\\
0 & *
\end {array}\right)\ \modulo\ N\right\}$$

$$K_p^1(N)=i_p^{-1}\left\{\gamma\in GL_2(\ZZ_p)\ |\ \gamma\equiv \left(\begin {array} {cc}
* & *\\
0 & 1
\end {array}\right)\ \modulo\ N\right\}.$$

%$$K_p^1(N)=i_p^{-1}\left\{\gamma\in GL_2(\ZZ_p)\ |\ \gamma\equiv \left(\begin {array} {cc}
%* & *\\
%0 & 1
%\end {array}\right)\ \modulo\ N\right\}.$$

We define
%$$V_1(N)=\prod_{p\not|N\ell}R_p^\s\s\prod_{p|N}K^1_p(N)\s(1+\pi_\ell R_\ell)$$ where $\pi_\ell$ is a uniformizer of $R_\ell$ and
$$V_0(N)=\prod_{p\not|N}R_p^\s\s\prod_{p|N}K_p^0(N).$$
$$V_1(N)=\prod_{p\not|N\ell}R_p^\s\s\prod_{p|N}K^1_p(N)\s(1+\pi_\ell R_\ell)$$ where $\pi_\ell$ is a uniformizer of $R_\ell.$

\noindent We introduce a Dirichlet character $\psi:(\ZZ/\Delta\ell N\ZZ)^\s\to \CC^\s$;  as a  general hypothesis we will assume that the order of  $\psi$ is  prime to $\ell$. By an abuse of notation, we will denote by $\psi$ also the adelization of the Dirichlet character $\psi$ and we will  denote by $\psi_p$ the composition of $\psi$ with the inclusion $\QQ_p^\s\to\A^\s$. We fix a regular character $\chi:\ZZ_{\ell^2}^\s\to\overline\QQ^\s$ such that $\chi|_{\ZZ_\ell^\s}=\psi_\ell|_{\ZZ_\ell^\s}$. We observe that $\chi$ is not uniquely determinated by $\psi$; we extend $\chi$ to $\QQ_{\ell^2}^\s$ by putting $\chi(\ell)=-\psi_\ell(\ell)$.
 If we fix an embedding of $\overline\QQ$ in $\overline\QQ_\ell$, then we can regard the values of $\chi$ in this field.\\
By local classfield theory, $\chi$ may be regarded as a character of $I_\ell$. %$$I_\ell\to\galois(\QQ_{\ell^2}^{ab}/\QQ_\ell^{unr})\stackrel{\sim}{\to}\ZZ_{\ell^2}^\s\stackrel{\chi}{\to}\overline\QQ_\ell^\s$$
Then there is a $\ell$-type $\tau$ associated to $\chi$ (for a reference see \cite{ConradDiamondTaylor99}), $\tau=\chi\oplus\chi^\sigma$ where $\chi^\sigma$ is the non trivial element of $\galois(\QQ_{\ell^2}/\QQ_\ell)$.\\
We put $\widehat\psi=\prod_{p|N}\psi_p\s\chi$; then $\widehat\psi$ is a character of $V_0(N)$ and its kernel is  $V_1(N)$ where

%\begin{displaymath}
%V_1(N)=\left\{\begin{array}{ll}
%\prod_{p\not|N\ell}R_p^\s\s\prod_{p|N}K^1_p(N)\s(1+\pi_\ell R_\ell) & \textrm{if $\psi$ is not trivial}\\
%\prod_{p\not|N\ell}R_p^\s\s(1+\pi_\ell R_\ell) & \textrm{if $\psi$ is trivial}
%\end{array}\right.
%\end{displaymath}

%if $\psi$ is not trivial then $$V_1(N):=\prod_{p\not|N\ell}R_p^\s\s\prod_{p|N}K^1_p(N)\s(1+\pi_\ell R_\ell)$$ with

%Let $\widehat\psi$ be a character  of $V_0(N)$ and we denote by $V_1(N)$ its kernel.
\noindent For $i=0,1$, we consider the adelic Shimura curve associated to the group $V_i(N)$: $$\XX_i(N)=B_\QQ^\s\setminus B^\s_\A/K_\infty^+\s V_i(N)$$ where $K_\infty^+=\RR^\times SO_2(\RR).$

%Let we write  $\Omega=V_0(N)/V_1(N)$ as a $\Omega_\ell\s\Omega$ where $\Omega_\ell$ is the $\ell$-Sylow subgroup of $\Omega$ and $\Omega_0$ is the subgroup of $\Omega$ with order prime to $\ell$.
\noindent Let $\OO$ be the ring of integer of a finite extension $K$ of $\QQ_\ell$ containing $\QQ_{\ell^2}$ and the image of $\widehat\psi$; let $k$ be its residue field.\\
Since $K_p^0(N)/K_p^1(N)\simeq (\ZZ/p\ZZ)^\s$ and $R_\ell^\s/(1+\pi_\ell R_\ell)\simeq \FF_{\ell^2}^\times$ we can regard $\ws=\prod_{p|N} \psi_p\times \chi$ as a character of the group  $V_0(N)/V_1(N)\cong (\ZZ/N\ZZ)^{\s}\s\FF_{\ell^2}^{\s}$ by isomorphims $i_p$. Then $V_0(N)/V_1(N)$ naturally acts on $H^1(\XX_1(N),\OO)$ via the projection $\XX_1(N)\to\XX_0(N)$. As a general hypothesis, we suppose that the order of the group $V_0(N)/V_1(N)$ is invertible in $\OO$, thus if we denote  by $H^1(\XX_1(N),\OO)^{\widehat\psi}$ the sub-Hecke-module of $H^1(\XX_1(N),\OO)$ on which $V_0(N)/V_1(N)$ acts by the character $\widehat\psi$, then $H^1(\XX_1(N),\OO)^{\widehat\psi}$ is a direct summand of $H^1(\XX_1(N),\OO)$.

\noindent Let $S$ be a square-free number prime to $N\Delta$; by abuse of notation we will denote by $S$ also the finite set of rational primes $p_j$ such that $\prod_jp_j=S$.
 For $i=0,1$  we define the group $$V_i(N,S)=\left\{u\in V_i(N)\ |\ u_q\equiv\left(\begin {array} {cc}
 * & *\\
 0 & *
 \end {array}\right)\ \modulo\ q\ {\rm for}\ q|S\right\}; $$ we observe that $V_0(N,S)=V_0(NS)$ and if $S=\emptyset$ then $V_i(N,S)=V_i(N)$. For any prime number $q$ not dividing $\Delta N$ let $\eta_q$ be the id\`ele  in $B_\A^\s$ defined by $\eta_{q,v}=1$ if $v\not=q$ and $\eta_{q,q}=i^{-1}_q\left(
\begin{array}
[c]{cc}%
q & 0\\
0 & 1
\end{array}
\right).$
Then both $V_i(N,q)$ and $\eta_qV_i(N,q)\eta_q^{-1}$ are subgroups of $V_i(N)$. The two injections (inclusion and conjugation by $\eta_q$) give rise  to degeneracy maps on the Shimura curves $\XX_i(N,q)\to\XX_i(N)$ where, for $i=0,1$ and $S$ as above, $\XX_i(N,S)$ is the Shimura curve associated to the group $V_i(N,S)$. The degeneracy maps commute to the action of $V_0(N)/V_1(N)$ so that there is a map of cohomology groups
\begin{equation}\label{eq:mappaalpha}
\alpha:  H^1(\XX_1(N),\OO)^{\widehat\psi} \oplus H^1(\XX_1(N),\OO)^{\widehat\psi}\to H^1(\XX_1(N,q),\OO)^{\widehat\psi}\end{equation}
defined by $\alpha=\Res\oplus |_{\eta_q}$ where $\Res$ is the usual restriction map and $ |_{\eta_q}$ is the map on cohomology induced by conjugation by $\eta_q$. \\
Essentially the open problem is to verify that the map $\alpha$ has torsion-free cokernel. We will refer to  this kind of problems as \lq\lq Problem of Ihara\rq\rq; the motivation for this name will be clear in the Section \ref{literature}.\\
One can naturally define an Hecke algebra acting over each group cohomology defined above (see \cite{Hida88}, \cite{Terracini03}, \cite{miri2006}). A simplified version of the problem of Ihara would be to establish the torsion-freenes property of the cokernel of the map $\alpha$ restricted to  non-Eisenstein components of the cohomology. Thus we formulate the following weaker conjecture:

{}

\begin{conjecture}\label{noscon}
\noindent Let $q$ be a prime number such that $q\not|N\Delta\ell$. We fix a maximal non Eisenstein ideal $\mm$ of the Hecke algebra  $\TT^{\widehat\psi}(N)$ acting on the group $H^1(\XX_1(N),\OO)^{\widehat\psi}$. Let $\TT^{\widehat\psi}(N,q)$ be the Hecke algebra acting on $H^1(\XX_1(N,q),\OO)^{\widehat\psi}$, let $\mm_q$ be the inverse image of $\mm$ under the natural map $\TT^{\widehat\psi}(N,q)\to\TT^{\widehat\psi}(N)$.
The map $$\alpha_\mm:H^1(\XX_1(N),\OO)_\mm^{\widehat\psi}\s H^1(\XX_1(N),\OO)_\mm^{\widehat\psi}\to H^1(\XX_1(N,q),\OO)_{\mm_q}^{\widehat\psi}$$  has a torsion free cokernel.

%is such that   $\alpha_\mm\otimes_\OO k$ is injective.

\end{conjecture}

\noindent It easily follows from the Hochschild-Serre spectral sequence that $$H^1(\XX_1(N,S),\OO)^{\widehat\psi}\simeq H^1(\XX_0(NS),\OO({\widehat\psi}))$$  where $\OO({\widehat\psi})$ is the sheaf $B^\s\setminus B_\A^\s\s\OO/K_\infty^+\s V_0(NS),$ $B^\s$ acts on $B^\s_\A\s\OO$ on the left by $\alpha\cdot(g,m)=(\alpha g,m)$ and $K_\infty^+\s V_0(NS)$ acts on the right by $(g,m)\cdot v=(g,m)\cdot(v_\infty,v^\infty)=(gv,{\widehat\psi}(v^\infty)m)$ where $v_\infty$ and $v^\infty$ are respectively the infinite and finite part of $v$.

\noindent For $i=0,1$, let $\Phi_i(N,S)=(GL_2^+(\RR)\s V_i(N,S))\cap B_\QQ^\s$.  We observe that $\Phi_0(N,S)=\Phi_0(NS)$.  By the isomorphism $i_\infty$, the group $\Phi_i(N,S)$ can be considered as a discrete subgroup  having finite covolume in $SL_2(\RR)$ \cite{Vigneras80}. Since we are supposing  $\Delta\not=1$, then $B$ is a division algebra, $SL_2(\RR)/\Phi_i(N,S)$ is compact and $\Phi_i(N,S)\setminus SL_2(\RR)/SO_2(\RR)$ is a compact Riemann surface.\\
By simplicity, we shall assume that the group $\Phi_0(NS)$ does not contain elliptic elements. If not, there is a standard technique to add a prime $s$ in the level such that $\Phi_0(NS)$ has not elliptic elements and the local Hecke component of the cohomology at level $NSs$ is the same that at level $NS$.\\
\noindent By translating to the cohomology of groups we obtain (see \cite{deSmitLenstra97}, Appendix) $H^1(\XX_1(N,S),\OO)^{\widehat\psi}\simeq H^1(\Phi_0(NS),\OO(\ts)),$ where $\ts$ is the restriction of ${\widehat\psi}$ to $\Phi_0(NS)/\Phi_1(N,S)$ and $\OO(\widetilde\psi)$ is $\OO$ with the action of $\Phi_0(NS)$ given by $a\mapsto\widetilde\psi^{-1}(\gamma)a$.
%A similar analysis holds if we  \lq\lq add\rq\rq\ the prime $q$ to the level, thus $H^1(\XX_1(N,q),\OO)^{\widehat\psi}\simeq H^1(\Phi_0(Nq),\OO(\ts))$ and
By this analysis, we can translate the map $\alpha$ in cohomology of groups. By strong approximation, write $\eta_q=\delta_qg_\infty u$, with $\delta_q\in B^\s$, $g_\infty\in GL_2^+(\RR)$ and $u\in V_0(Nq)$.
%By abuse of notation we will call $\alpha$ the map  $$\alpha:H^1(\Phi_0(N),\OO(\ts))\otimes H^1(\Phi_0(N),\OO(\ts))\to H^1(\Phi_0(N,q),\OO(\ts)).$$
Consider first the map
\begin{eqnarray}
|_{\eta_q}:H^1(\Phi_0(N),\OO(\tp)) & \to & H^1(\Phi_0(N,q),\OO(\tp))\nonumber\\
x   & \longmapsto & x|_{\eta_q}\nonumber
\end{eqnarray}
working as follows: let $\xi$ be a cocycle representing the cohomology class $x$ in $H^1(\Phi_0(N),\OO(\tp))$; then $x|_{\eta_q}$ is represented by the cocycle $\xi'(\gamma)=\widehat\psi^{-1}(u)\cdot\xi(\delta_q\gamma\delta_q^{-1}).$ It is easy to see that this definition does not depend on the choice of the approximation. Define now:
\begin{eqnarray}
\alpha:H^1(\Phi_0(N),\OO(\tp))\s H^1(\Phi_0(N),\OO(\tp))&\to&H^1(\Phi_0(N,q),\OO(\tp))\nonumber\\
(x\ \ \ \ \ \ ,\ \ \ \ \ \ y)\ \ \ \ \ \ \ \ \ \ \ & \longmapsto & \Res(x)+y|_{\eta_q}\nonumber
\end{eqnarray}
where $\Res$ is the restriction map $\Res_{\Phi_0(N)/\Phi_0(N,q)}$ from $\Phi_0(N)$ to $\Phi_0(N,q)$.

\noindent We observe that by Nakayama's lemma,  Conjecture \ref{noscon} is equivalent to prove that the map $$\alpha_\mm\otimes k:H^1(\Phi_0(N),k(\tp))_\mm\s H^1(\Phi_0(N),k(\tp))_\mm\to H^1(\Phi_0(N,q),k(\tp))_{\mm_q}$$ is injective.
% where $H^1(\Phi_0(N,S),k(\tp))_\mm=H^1(\Phi_0(N,S),\OO(\tp))_\mm\otimes_\OO k$.\\

\section{The literature about the problem of Ihara}\label{literature}

\noindent Besides the intrinsic interest of Ihara's problem, it turned out to be an essential tool to solve problems of raising the level of modular forms. For this reason, it has been generalized in many ways, and the literature about this problem is very rich. In this section we will present the most relevant results about this subject.

\subsection{The modular case}\label{modular}

\noindent Let us consider a quaternion algebra $B$ with discriminant $\Delta=1$. Under this assumption it is known that  $B\cong M_2(\QQ)$. Let $N$ be a positive integer number and let $q$ be a prime number $q\not|N$. We consider the modular curve $X(N)$ (resp. $X_1(N)$) associated to the congruence subgroup $\Gamma(N)$ (resp. $\Gamma_1(N)$) and let $X(N,q)$ (resp. $X_1(N,q)$) be the modular curve associated to  $\Gamma(N)\cap\Gamma_0(q)$ (resp. $\Gamma_1(N)\cap\Gamma_0(q)$) \cite{Shimura71}. In 1983, Ribet \cite{Ribet83}, observed that for all prime numbers $\ell$, the problem of showing that the map $$\alpha\otimes_\OO k:H^1(X(N),\ZZ/\ell\ZZ)\oplus H^1(X(N),\ZZ/\ell\ZZ)\to H^1(X(N,q),\ZZ/\ell\ZZ)$$ is injective is just a special case of Lemma 3.2 of Ihara \cite{Ihara73}. For this reason we are referring to this kind of results as the \lq\lq problem of Ihara\rq\rq. In \cite{Ribet83} Ribet gives a more direct proof of the injectivity of $\alpha\otimes_\OO k$ working with the  parabolic cohomology groups . The key of his proof is the knowledge of the amalgamamated product $$\Gamma=\Gamma(N)\ast_{\Gamma(N,q)}\left(\left(
\begin{array}
[c]{cc}%
q & 0\\
0 & 1
\end{array}
\right)^{-1}\Gamma(N)\left(
\begin{array}
[c]{cc}%
q & 0\\
0 & 1
\end{array}
\right)\right)$$  and the injectivity of $\alpha\otimes_\OO k$ can be deduced (via the exact sequence of Lyndon \cite{Serre80}) by the vanishing of the parabolic cohomology group of $\Gamma$ with coefficients in $\FF_\ell$.  \\

\noindent In \cite{ConradDiamondTaylor99} Conrad-Diamond and Taylor prove the problem of Ihara related to a map $\alpha_\mm$ obtained by considering the localization of the weight $n$ cohomology to a  non-Eisenstein maximal ideal of the Hecke algebra acting on it. They require that the prime $q$ does not divide $N\ell$. Also in this case, the key of the proof is the knowledge of the amalgamated product $\Gamma$ and its  congruence subgroup property \cite{Serre970}. \\

\noindent In \cite{Diamond91},  Diamond  generalizes the work of Ribet to an arbitrary weight $n$ cohomology, under the hypothesis that $q\not|N(n-2)!$.  Also in this case the argument is the vanishing of the parabolic cohomology group of the amalgama $\Gamma$.  \\

\subsection{The quaternionic case}\label{quat}

\noindent We now consider an indefinite quaternion algebra $B$ over $\QQ$ of discriminant $\Delta>1$. The group-theoretic approaches working for elliptic modular curves seem to be inadequate for Shimura curves. In particular, the congruence subgroup property for $q$-arithmetic subgroups of $B^\s$ is conjectured, but unfortunately not known.\\

\noindent In \cite{DiamondTaylor94}, Diamond and Taylor  prove a problem of Ihara for Shimura curves in a special non-Eisenstein case, using a Galois theoretic argument.  Their make the following hypotheses:

\begin{enumerate}
	\item the prime $\ell$ does not divide the discriminant $\Delta$ of $B$;
	\item they work with Shimura curves arising from an open compact subgroup $U\subset B_\A^\s$ sufficiently small: it must be  contained in  the group $$V(N):=\prod_{p\not|N}R_p^\s\s\prod_{p|N}K^1_p(N)$$ %where $K^1_p(N)$ is defined as in Section \ref{sect:congruence}
	for $N$ prime to $\Delta$.
\end{enumerate}

%of $\prod_pR_p^\s$ consisting of those $u$ conguent to $\left(
%\begin{array}
%[c]{cc}%
%* & *\\
%0 & 1
%\end{array}
%\right) \modulo\ N$ for $N$ prime to $\Delta$.

\noindent We observe that  hypotheses 1. and 2., assure that the Shimura curve associated to such a group $U$, have good reduction at $\ell$, so they give a new argument for results about the problem of Ihara, which uses the arithmetic geometry of the curves.\\
We observe that the reason for which the approach of Diamond-Taylor fails in our case, is that under our hypotheses, the Shimura curves $\XX_i(N,S)$ do not have good reduction at $\ell$ (see for example \cite{Carayol86}). \\

\section{A reformulation in terms of the subgroup congruence property}

\noindent If $G$ is a subgroup of $B_{\bf A}^{\times}$, we denote by $G^{(1)}$ the subgroup of reduced norm $1$ elements of $G$.

For $p\not=\ell$ we denote by $T_p$ the Hecke operator at $p$ acting on $H^1(\Phi_0(N),\OO(\tilde\psi))$ and $H^1(\Phi_0(N),k(\tilde\psi))$ and by $U_p$ the Hecke operator at $p$ acting on $H^1(\Phi_0(N),\OO(\tilde\psi))$ and $H^1(\Phi_0(Np),k(\tilde\psi))$, see for example \cite{Hida88}.
We observe that for $p$ not dividing $N\Delta$
\begin{equation}\label{eq:corestr} T_p(x)={\rm Cor}_{\Phi_0(N)/\Phi_0(Np)}(x|_{\eta_p})\end{equation}

\noindent The following properties of Hecke operators are well-known, and easily proved by straighforward calculations:

\begin{Propo1}
Let $p,q$ be primes not dividing $N\Delta\ell$. For every $x\in H^1(\Phi_0(N),\OO(\tp))$ the following equalities hold:
$$T_p(x|_{\eta_q})=T_p(x)|_{\eta_q}\ \ \ {\rm if}\ p\not =q,$$
$$T_p(\Res_{\Phi_0(N)/\Phi_0(Nq)}(x))=\Res_{\Phi_0(N)/\Phi_0(Nq)}(T_p(x))\ \ \ {\rm if}\ p\not =q,$$
\begin{equation}  \label{r}
U_p(x|_{\eta_p})=p\psi(p)\Res_{\Phi_0(N)/\Phi_0(Np)}(x),
\end{equation}
\begin{equation}  \label{s}
U_p(\Res_{\Phi_0(N)/\Phi_0(Np)}(x))=\Res_{\Phi_0(N)/\Phi_0(Np)}(T_p(x))-x|_{\eta_p}.
\end{equation}
\end{Propo1}

\begin{Propo1}\label{prop:annullatori}
Let $p$ be a prime not dividing $N$, and $x_1,x_2\in H^1(\Phi_0(Np),k(\tilde\psi))$. Suppose that \begin{equation}\label{eq:ipot} Res_{\Phi_0(N)/\Phi_0(Np)}(x_1)= x_2|_{\eta_p}.\end{equation}
 Then
\begin{itemize}
\item[a)] $T_p(x_2)=(p+1)x_1$.
\item[b)] $T_p(x_1)=\psi(p)(p+1)x_2$.
\item[c)] The operator $T_p^2-\psi(p)(p+1)^2$ annihilates $x_1$ and $x_2$.
\end{itemize}
\end{Propo1}

\begin{proof}
We obtain (a) directly applying corestriction to (\ref{eq:ipot}). By applying $U_p$ followed by corestriction to (\ref{eq:ipot}), we get
$pT_p(x_1)=p\psi(p)(p+1)x_2$ and thus (b), since $p\not=\ell$. Claim (c) is an immediate consequence of (a) and (b).
\end{proof}

\begin{corollary} Let $M$ be an integer and suppose that (\ref{eq:ipot}) holds for every prime $p$ not dividing $MN\Delta$. Then the annihilator of $x_1$ (and $x_2$) is contained in an Eisenstein ideal.
\end{corollary}

\noindent If $U$ is an open compact subgroup of $B_{\bf A}^{\times,\infty}$, we put $\Phi(U)=GL_2^+(\RR)U\cap B^\s$.\\
Let $M$ be an integer and define the compact open subgroup of $B_{\bf A}^{\times,\infty}$
$$U(M)=\prod_{p\not | M} R_p^\times \times \prod_{p|M}(1+MR_p).$$
\begin{definition} A subgroup $\Psi$ of $\Phi_0(M)$ is a congruence subgroup if it contains $\Phi(U(M))$ for some integer $M$.
\end{definition}

\begin{Propo1}\label{prop:iniettivitarestr} Let $\Psi$ be a congruence subgroup of $\Phi_0(N)$. Then
\begin{itemize}
\item[a)]
the kernel of the restriction
$$\Res_{\Phi_0(N)/\Psi}:H_1(\Phi_0(N),k(\tilde\psi))\to H_1(\Psi,k(\tilde\psi))$$
is stable for Hecke operators $T_p$, for almost every prime $p$.
\item[b)] let $\mm$ be a non Eisenstein maximal ideal of $\TT^{\hat\psi}(N)$. Then the restriction
 $${\rm Res}_{\Phi_0(N)/\Psi}:H^1(\Phi_0(N),k(\tilde\psi))_\mm\longrightarrow H^1(\Psi,k(\tilde\psi))$$
 is injective.
 \end{itemize}
\end{Propo1}
\begin{proof}
\begin{itemize} \item[a)] Suppose that $\Psi\supseteq \Phi(U(M))$ and let $p$ be a prime not dividing $N\ell M$. By strong approximation we can decompose $\eta_p=\delta_p g_\infty u$ with $\delta_p\in B^\times$, $h_\infty\in{\rm GL}_2^+(\RR)$ and $u\in V_0(Np)$.
 %We may choose $u$ in such a way that $\delta_px\delta_p^{-1}x^{-1}\subseteq \Psi$ if $x\in\Psi$.
 Observe that since $p\not |M$, $\Phi$ is dense in $\Phi_0(N)$ with respect to the $p$-adic topology, so that the map  $\Psi \to \Phi_0(N)/\Phi_0(Np)$ is surjective. Then we can choose in $\Psi$ a set of representatives $\{h_1,...,h_s\}$ of right cosets of $\Phi_0(Np)$ in $\Phi_0(N)$, and we get a decomposition of double cosets $\Phi_0(N)\delta_p\Phi_0(N)=\coprod_i\Phi_0(N)\delta_p h_i$
 and $\Psi\delta_p\Psi=\coprod_i\Psi\delta_p h_i$.
 Let $\xi$ be a cocycle representing a cohomology class in ${\rm H^1}(\Phi_0(N),k(\tilde\psi))\mm $ which restricts to zero  on $\Psi$, and let $\gamma\in\Psi$. Write $\delta_ph_i\gamma=\gamma_i\delta_ph_{j(i)}$ with $\gamma_i\in \Psi$. By definition $(T_p\xi)(\gamma)=\sum_i\psi(\delta_ph_i)\xi(\gamma_i)=0$.
 \item[b)] We can suppose that $\Psi=\Phi(U(M))$ with $N\ell$ dividing $M$, so that $\tilde\psi$ is trivial over $\Psi$. Let $\xi$ be a cocycle representing a cohomology class in ${\rm H^1}(\Phi_0(N),k(\tilde\psi))_\mm$ which restricts to zero  on $\Psi$ and let $\gamma\in\Phi_0(Np)$. Since $\Psi$ is normal in $\Phi_0(N)$, we can write $h_i\gamma=\gamma h'_ih_{j(i)}$, with $h'_i\in \Psi\cap\Phi_0(Np)$; then  $\delta_ph'_i\delta_p^{-1}\in \Psi$ and
$(T_p\xi)(\gamma)=\psi(\delta_p)\sum_i\xi(\delta_p\gamma h'_i\delta_p^{-1})=\psi(\delta_p)\sum_i\xi(\delta_p\gamma\delta_p^{-1})=(p+1)\xi|_{\eta_p}(\gamma)$.  Then we get $\Res_{\Phi_0(N)/\Phi_0(Np)}(T_p\xi)=(p+1)\xi|_{\eta_p}$. By  Proposition \ref{prop:annullatori}b) $T_p^2-\psi(p)(p+1)^2$ annihilates $\xi$, so it must belong to $\mm$ unless $\xi=0$. But $\mm$ is non Eisenstein, so it cannot contain $T_p^2-\psi(p)(p+1)^2$ for almost every $p$. Therefore $\xi$ must be zero.

 \end{itemize}
\end{proof}

Let us consider the amalgamated product:
\begin{equation}\label{amalgama}
\Phi:=\Phi_0(N)\ast_{\Phi_0(Nq)}\delta_q^{-1}\Phi_0(N)\delta_q.
\end{equation}
%The structure of $\Phi$ is not known; just the following inclusion comes directly from the definition of  amalgamated product:
The following proposition describes the structure of $\Phi$:

\begin{Propo1}
Let $R(N)$ be an Eichler order of level $N$ in $B$ and let $q$ be a prime not dividing $N$.  Then $\Phi=\left(R(N)\otimes_\ZZ\ZZ\left[\frac{1}{q}\right]\right)^{(1)}$.

\end{Propo1}

%In particular, the amalgamated product $\Phi$ is a subgroup of finite index of  $\left(R\otimes_\ZZ\ZZ\left[\frac{1}{q}\right]\right)^{(1)}$.

\begin{proof}
If $q$ is a prime number not dividing $\Delta N$ then  $R_q(N)^{(1)}\cong SL_2(\ZZ_q)$ and, by \cite{Serre80} p.79 $$R_q(N)^{(1)}\ast_{R_q(Nq)^{(1)}}\delta_q^{-1} R_q(N)^{(1)}\delta_q\cong B_q^{(1)}.$$ Since $\Phi_0(N)\cap R_q(Nq)^{(1)}=\delta_q^{-1}\Phi_0(N)\delta_q\cap R_q(Nq)^{(1)}=\Phi_0(Nq)$, by Proposition 3. \S1.3 of \cite{Serre80}, there is an injection $\Phi\hookrightarrow B_q^{(1)}.$ The universal property of the amalgamated product then assures that there is an inclusion $\Phi\subseteq  \left(R(N)\otimes_\ZZ\ZZ\left[\frac{1}{q}\right]\right)^{(1)}$. By strong approximation, $\Phi_0(N)$ is dense in $R_q(N)^{(1)}$ with respect to the $q$-adic topology, so that $\Phi$ is dense in $B_q^{(1)}$, and therefore in $\left(R(N)\otimes_\ZZ\ZZ\left[\frac{1}{q}\right]\right)^{(1)}$. But $\Phi_0(N)=\SL_2(\ZZ_q)\cap \left(R(N)\otimes_\ZZ\ZZ\left[\frac{1}{q}\right]\right)^{(1)}$, so that $\Phi$ is also open (and thus closed) in $\left(R(N)\otimes_\ZZ\ZZ\left[\frac{1}{q}\right]\right)^{(1)}$. This implies $\Phi=\left(R(N)\otimes_\ZZ\ZZ\left[\frac{1}{q}\right]\right)^{(1)}$. \end{proof}

If the amalgama $\Phi$ has the subgroup congruence property, then the same approach employed by Ribet and Diamond in the classical case is applicable to the quaternionic case:

\begin{theorem} Assume that $\Phi$ has the subgroup congruence property. Then Conjecture \ref{noscon} is true.
\end{theorem}

\begin{proof}It suffices to show that if $(x,y)$ is in the kernel of $\alpha_\mm$ then $x$ restricts to zero in some congruence subgroup of $\Phi_0(N)$. If this is the case, we would have $y|\eta_q=0$, and applying the operator $U_p$ this yields $\Res_{\Phi_0(N)/\Phi_0(Nq)}(y)=0$, and this implies $x=0,y=0$ by Proposition \ref{prop:iniettivitarestr}.b).
We can consider the  the Lyndon exact sequence
$$H^1(\Phi,k(\tilde\psi))\longrightarrow H^1(\Phi_0(N),k(\tilde\psi))^2\buildrel {\alpha\otimes k} \over\longrightarrow H^1(\Phi_0(Nq),k(\tilde\psi))$$
\cite[II,2.8]{Serre80}. Then $(x,y)$ comes from an element $z\in H^1(\Phi,k(\tilde\psi))$. Choose a subgroup $\Psi'$ of finite index in $\Phi$ such that $z$ restricts to zero in $\Psi'$. By our assumption, $\Psi'$ is a congruence subgroup, so that $x$ restricts to zero in $\Psi=\Psi'\cap\Phi_0(N)$ which is a congruence subgroup of $\Phi_0(N)$.
\end{proof}

\noindent Unfortunately the congruence subgroup property for $q$-arithmetic subgroups of $B^\s$ is conjectured, but  not known. A complete reference about the  present status of the congruence subgroup problem is \cite{Rapinchuk99}.

Let $\Sigma$ be the set of places $\{q,\infty\}$ with $q$ a prime number as in the previous sections; we observe that  $\Sigma$ does not contain any anisotropic nonarchimedean place for $B^{(1)}$.
Let us consider on $B^{(1)}$ the $\Sigma$-arithmetic topology $\tau_a$ (generated by all subgroups of finite index in $\Phi$) and $\tau_c$ the $\Sigma$-congruence topology (generated by the $\Sigma$-congruence subgroups in $\Phi$)
(see, for example \cite{Serre67}). We consider the natural continuous homomorphism $\pi:(\overline{B^{(1)}},\tau_a)\to(\overline{B^{(1)}},\tau_c)$ where  $(\overline{B^{(1)}},\tau_a)$ and $(\overline{B^{(1)}},\tau_c)$ are the completions of $B^{(1)}$ with respect to $\tau_a$ and $\tau_c$. The kernel of $\pi$ is the so called $\Sigma$-congruence kernel, denoted $C(\Sigma)$ and its size measures deviation from the affirmative solution of the congruence subgroup problem.

\noindent Since the normal subgroup structure of $B^{(1)}$ is given by  the Platonov-Margulis conjecture  (as proved in \cite{Segev99}), by Theorem 1 of \cite{Rapinchuk99}, if  $C(\Sigma)$ is central then   $\Phi$ has the congruence subgroup property. Thus,  Conjecture \ref{noscon} would be a consequence of  the following:

\begin{conjecture} The $\Sigma$-congruence kernel $C(\Sigma)$ is central. \end{conjecture}

As pointed out in \cite{Rapinchuk99}, this last open problem still does not have a uniform approach.

%The centrality of the $\Sigma$-congruence kernel may be seen as a reformulation of our Conjecture \ref{noscon}.

\noindent By Theorem 5 of \cite{Rapinchuk99}, if $\Phi$ has the bounded generation (BG) (i.e. there are elements $\gamma_1,...,\gamma_t\in\Phi$ such that $\Phi=\langle\gamma_1\rangle...\langle\gamma_t\rangle$ where $\langle\gamma_i\rangle$ is the cyclic subgroup generated by $\gamma_i$), then the $\Sigma$-congruent kernel is central. Unfortunately (BG) is not established for $\Phi$ but only conjectured \cite{Rapinchuk99}.  \\
Our optimism is renforced by the fact that the problem of showing sufficient conditions for the centrality of the $\Sigma$-congruence kernel for general algebraic groups, is the subject of research of many mathematicians, see for example \cite{PrasadRapinchuk08}, \cite{RapinPra}, \cite{PraRag}.

\section{Some consequences of the conjecture}\label{sect:conseguenze}

By Jacquet-Langlands correspondence and Matsushima-Shimura isomorphism \cite{MatsushimaShimura63}, some modular forms can be reinterpreted as elements of the cohomology of Shimura curves (for details see \cite{MoriTerracini98}, \cite{Terracini03}). Therefore one can use quaternionic cohomological modules in the construction of Taylor-Wiles systems for certain deformation problems. This has be done for example in \cite{Terracini03}, \cite{miri2006}.\\
In this section we will describe two interesting consequences of  Conjecture \ref{noscon}.

\noindent Let $S$ be a finite set of rational primes which does not divide $N\Delta$.
Keeping the same notations as in the previous sections, let $S_2(V_1(N,S))$ be the space of quaternionic automorphic forms  right invariant by $V_1(N,S)$ \cite{Hida88}; let  $S_2(V_0(NS),\widehat\psi)$ be the subspace of $S_2(V_1(N,S))$ consisting of the  $\varphi$ such that $$\varphi(gk)=\widehat\psi(k)\varphi(g)\ \ \ {\rm for\ any\ }k\in V_0(NS),\ g\in B_\A^\s.$$ The Jacquet-Langlands correspondence establishes an isomorphism $$JL:S_2(V_0(NS),\widehat\psi)\tilde{\to}V_{\widehat\psi}$$ where $V_{\widehat\psi}$ is a direct summand of $S_2(\Gamma_0(S\Delta')\cap\Gamma_1(N\ell^2))$ consisting of classical modular forms of level 2, nebentypus $\psi$, special at primes dividing $\Delta'$ and supercuspidal of type $\tau=\chi\oplus\chi^\sigma$ at $\ell$.
By the Matsushima-Shimura isomorphism,  one can see such forms  in the $\ell$-adic cohomology of a  Shimura curve, as we shall show in details in the following.

\subsection{The Hecke structure of the cohomology of Shimura curves}
\noindent Let us fix $f\in S(\Gamma_0(N\Delta'\ell),\psi)$ a modular newform of weight 2, level $N\Delta'\ell $, supercuspidal of type $\tau$ at $\ell$, special at primes dividing $\Delta'$  and with Nebentypus $\psi$ of order prime to $\ell$.  Let $\rho$ be the Galois representation associated to $f$ and let $\orho$ be its reduction modulo $\ell$; we impose the following conditions on $\orho$:
 \begin{equation}\label{con1}
\orho\ {\rm is\ absolutely\ irreducible};
\end{equation}
\begin{equation}\label{cond2}
{\rm if}\ p|N\ {\rm then}\ \orho(I_p)\not=1;
\end{equation}
\begin{equation} \label{rara}
{\rm if}\ p|\Delta'\ {\rm and}\ p^2\equiv 1\ \modulo\ \ell\ {\rm then}\ \orho(I_p)\not=1;
\end{equation}
\begin{equation}\label{end}
\End_{\overline\FF_\ell[G_\ell]}(\orho_\ell)=\overline\FF_\ell.
\end{equation}
\begin{equation}\label{c3}
{\rm if}\ \ell=3,\ \orho\ \ {\rm is\ not\ induced\ from\ a\ character\ of\ }\ \QQ(\sqrt{-3}).
\end{equation}
\noindent We denote by $\Delta_1$ the product of primes $p|\Delta'$ such that $\orho(I_p)\not=1$ and by $\Delta_2$ the product of primes $p|\Delta'$ such that $\orho(I_p)=1$ so that we are assuming that $p^2\not\equiv 1\ \modulo\ \ell$ if $p|\Delta_2$. We say that the representation $\rho$ is of {\bf type $(S,{\rm sp},\tau,\psi)$} if $\rho$ is a deformation of $\bar\rho$ and the following conditions hold:
\begin{itemize}
\item[a)] $\rho$ is unramified outside $N\Delta S$.
\item[b)] if $p|\Delta_1N$ then $\rho(I_p)\simeq\orho(I_p)$ ;
\item[c)] if $p|\Delta_2$ then $\rho_p$ satisfies the condition $\tr(\rho(F))^2=\psi(p)(p+1)^2$ for a lift $F$ of $\Frob_p$ in $G_p$;
\item[d)] $\rho_\ell$ is weakly of type $\tau$ (see \cite{ConradDiamondTaylor99});
\item[e)] $\det(\rho)=\epsilon\psi$, where $\epsilon:G_\QQ\to\ZZ_\ell^\s$ is the cyclotomic character.
\end{itemize}

\noindent As proved in \cite{Terracini03} or \cite{miri2006}, this is a deformation condition. Let $\mathcal R_S$ be the universal deformation ring parametrizing representations of type $(S,{\rm sp},\tau,\psi)$ with residual representation $\orho$.

We say that a modular eigenform is of {\rm type} $(S,sp,\tau,\psi)$ if the Galois representation associated $\rho_f$ is of type $(S,sp,\tau,\psi)$.

Let $\mathcal{B}_S$ denote the set of newforms $h$ of level dividing $N\Delta' S \ell^2$ which are of type $(S,sp,\tau,\psi)$ and for each $h\in \mathcal{B}_S$ let $\OO_h$ be the extension of $\OO$ generated by the Hecke eigenvalues of $h$.  Let $\TT^\psi_S$ be the Hecke algebra acting on the space of modular forms of type $(S,{\rm sp},\tau,\psi)$. Let $\TT_S$ denote the sub-$\OO$- algebra of $\prod_{h\in \mathcal{B}_S} \OO_h$ generated by the elements $T_p=(a_p(h))_{h\in \mathcal{B}_S}$ for $p\not | N\Delta 'S\ell$. Similarly to \cite[Proposition 3.3]{Terracini03}, we can see that there is a surjective homomorphism $\mR_S\to \TT_S$. By combining the Jacquet-Langlands correspondence with the discussion in Section 4.2 of [DDT], we know that $\TT_S$ is isomorphic to the localization in a maximal ideal $\mm_S$ of the full Hecke algebra $\TT^{\hat\psi}(N,S)$ acting on $H^1(\XX_1(N,S),\OO)^{\hat\psi}$, and that $\mm_S$ lies over $\mm=\mm_\emptyset$. We define $\MM_S=H^1(\XX_1(N,S),\OO)^{\hat\psi}_{\mm_S}$.
As explained in Section \ref{Sect:duality} each $\MM_S$ is endowed by a $\TT_S$-bi-linear, $\OO$-perfect pairing $\langle\cdot,\cdot\rangle_S$.
%If $\psi$ is trivial, then $$V_1(N):=\prod_{p\not|NS\ell}R_p^\s\s(1+\pi_\ell R_\ell)$$
%We have proved that if $S=\emptyset$, then theorem \ref{goal holds, where $\mR_\emptyset=\mR$ and $\TT^\psi_\emptyset=\TT^\psi$.\\

For simplicity, the following application of our main conjecture is stated for sets $S$ of primes not congruent to $\pm 1$ modulo $\ell$:
\begin{theorem}\label{conj2}
Suppose that Conjecture \ref{noscon} is true. Let $S$ be a set of primes $q$ not dividing $N\Delta$ such that $q^2\not\equiv 1\  \modulo\ \ell$. Then
\begin{enumerate}
\item $\mathcal R_S\to \TT_S$ is an isomorphism of complete intersection rings;
\item $\mathcal M_S$ is a free $\TT_S$-module.
\end{enumerate}
\end{theorem}
The proof relies over a theorem of Diamond \cite[Theorem 2.4]{Diamond97}, that we briefly illustrate below.\par

We assume that $\OO$ is a complete discrete valuation ring with uniformizer $\lambda$ and residue field $k$. Fix complete local noetherian $\OO$-algebras $R,R'$ with $\OO$-algebra homomorphisms $\varphi:R'\to R$, $\pi:R\to \OO$; define $\pi'=\pi\circ \varphi$.\par
Let $M$ (resp. $M'$) be an $R$ (resp. $R'$) module, finitely generated and free over $\OO$; we can view $M$ as an $R'$-module via $\varphi$. Suppose that there exists an injective $R'$-homomorphism $\alpha:M\to M'$  with cokernel torsion free. Define $T=R/Ann_R(M), T'=R'/Ann_R'(M')$; then $\varphi$ induces a surjective map $T'\to T$.\par
Suppose that $\gp=\ker(\pi)$ is in the support of $M$ and let $\gp'=\varphi^{-1}(\gp)$.
Let $I$ (resp. $I'$) be the annihilator of $\gp$ (resp. $\gp'$) in $T$ (resp. $T'$). Then $\alpha$ induces an isomorphism $M[\gp]\simeq M'[\gp']$, which are both free $\OO$-modules of finite rank $d$. Define
\begin{eqnarray*}
\Omega &=& \frac {M}{M[\gp]+ M[I]}\\
\Omega' &=& \frac {M'}{M'[\gp']+ M'[I']}
\end{eqnarray*}

\noindent According to \cite[Theorem 2.4]{Diamond97}, if $\Omega$ has finite lenght over $\OO$, then the following statements are equivalent:
 \begin{itemize}
 \item[i)] $rk_\OO(M)\leq d\cdot rk_\OO(T)$ and $length_\OO(\Omega)\geq d\cdot lenght_\OO(\gp/{\gp^2})$;
 \item[ii)] $R$ is a complete intersection and $M$ is free (of rank $d$) over $R$.
 \end{itemize}

\noindent Diamond theorem allows, in some particular case, to deduce property $ii)$ for $M',R'$ once it is proved for $M,R$. We briefly explain this technique in the following.\\
\noindent Suppose moreover that $M$ (resp. $M'$) is equipped with a $T$ (resp. $T'$) bi-linear, $\OO$-perfect pairing $\langle\cdot,\cdot\rangle$ (resp. $\langle\cdot,\cdot\rangle'$) and let $\beta:M'\to M$ be the transpose of $\alpha$ w.r.t. these pairings. It is easy to see that $\Omega$ (resp. $\Omega'$) is isomorphic to the cokernel of the map $M[\gp]\to \Hom(M[\gp],\OO)$ (resp. $M'[\gp']\to \Hom(M'[\gp'],\OO)$, so that $\Omega$ and $\Omega'$ have finite lenght over $\OO$.
Suppose that $R$ is complete intersection and $M$ is free over $R$. Suppose moreover that
\begin{equation}\label{eq:disu} rk_\OO(\Omega')- rk_\OO(\Omega)\geq d\cdot (lenght_\OO(\gp'/{\gp'^2})-lenght_\OO(\gp/{\gp^2})) \end{equation}
Then by Diamond theorem we can conclude that $R'$ is complete intersection and $M'$ is free over $R'$.

\noindent If $x_1,...,x_d$ is a $\OO$-basis for $M[\gp]$ it is easy to see that
\begin{displaymath}
rk_\OO(\Omega)=v_\lambda\det
\left(\langle x_i,x_j\rangle\right)\end{displaymath}
and the same holds for $\Omega'$ and any basis $x'_1,...,x'_d$ of $M'[\gp']$. Since $\alpha(x_1),...,\alpha(x_d)$ is a basis of $M'[\gp']$, then:
\begin{displaymath}
rk_\OO(\Omega')=v_\lambda\left(\det
\left(\langle\alpha(x_i),\alpha(x_j)\rangle\right)\right)\end{displaymath}
If the composition $\beta\circ\alpha$ over $M[\gp]$ is the multiplication by an element $c$ of $\OO$, we have:

\begin{eqnarray*}
rk_\OO(\Omega')&=&v_\lambda\left(\det
\left(\langle x_i,\beta\circ\alpha(x_j)\rangle\right)\right)\\
&=&v_\lambda\left(\det
\left(\langle x_i,x_j \rangle\right)c^d\right)=d\cdot v_\lambda(c)+rk_\OO(\Omega)\end{eqnarray*}

\noindent On the other hand $\ker(\varphi)\subseteq \gp'$ and its image $K$ in $\gp'/\gp'^2$ is an $\OO$-module which generates the kernel of the map $\gp'/\gp'^2\to \gp/\gp^2$ induced by $\varphi$. Therefore
$$lenght_\OO(\gp'/{\gp'^2})-lenght_\OO(\gp/{\gp^2})= lenght_\OO(K),$$
and (\ref{eq:disu}) will be proved if we can show that
\begin{equation}\label{eq:disu1}
v_\lambda(c)\geq lenght_\OO(K).\end{equation}

%Let $T,T'$ be complete local noetherian $\OO$-algebras with an homomorphism $\varphi: T'\to T$. Let %$M$ (resp. $M'$) be a free $\OO$-module with a faithful action of $T$ (resp $T'$) and let %$\alpha:M\to M'$ be a $\OO$-homomorphism compatible with $\varphi$ with cokernel torsion free.
\subsection{A duality result on the cohomology of Shimura curves}\label{Sect:duality}

In this section we prove the existence
of a bilinear perfect Hecke-equivariant pairing $\langle \cdot,\cdot \rangle_S$ on $H^1(\XX_0(N,S),\OO(\ws))$; it induces an isomorphism $\mathcal M_S\to\Hom_\OO(\mathcal M_S,\OO)$ as Hecke modules; this is one of the ingredients of the proof of Theorem \ref{conj2} illustrated above.\\

The $\OO-$linear map
$b:\OO(\ws)\otimes_{\OO}\OO(\ws^{-1})\to\OO$
 allows to consider the cup product:
 \begin{eqnarray*}
 b^\s_S:H^1(\XX_0(N,S),\OO(\ws))\otimes_{\OO} H^1(\XX_0(N,S),\OO(\ws^{-1}))&\stackrel{\cup}{\to}&H^2(\XX_0(N,S),\OO))
 \simeq \OO
 \end {eqnarray*}

\begin{Propo1}\label{perfect}
The pairing $b_S^\s$ is perfect.
\end{Propo1}

\proof
It suffices to show that the cup product $$H^1(\XX_0(N,S),k(\ws))\otimes_{\OO} H^1(\XX_0(N,S),k(\ws^{-1}))\to k$$ is non-singular.\\
Under our hypothesis $\Delta\not=1$, the group $\Phi_0(N,S)$ is co-compact in $SL_2(\RR)$ (see for example \cite{Vigneras80}); then by Poincar\'e duality $$H^1(\XX_0(N,S),k(\ws))\simeq H_1(\XX_0(N,S),k(\ws)).$$ This isomorphism identifies cup and cap products; the cap product is non-singular by Theorem 10.5, \S V of \cite{Bredon97}.
\qed

%\noindent For every prime $p$ we denote by $\tilde T_p$ the operator on $H^1(\XX_0(N,S),\OO(\ws^{-1}))$ which is the adjoint of $T_p$ with respect to $b_S^\ast$. \par

\noindent Let $\omega_\ell$ be an element in $R_\ell^\times$ such that $\omega_\ell^2=\ell$ and $\chi(\omega_\ell\alpha\omega_\ell^{-1})=\chi^{-1}(\alpha)$ for every $\alpha \in R_\ell^\times$. Let us consider the element $\omega_{NS}$ of $B_\A^{\s,\infty}$ defined as follows: $\omega_{NS,\nu}=1$ if $\nu\not|NS\ell$, $\omega_{NS,\ell}=\omega_\ell$, $\omega_{NS,\nu}=i_\nu^{-1}\left(\begin{array}
[c]{cc}%
0 & -1\\
NS & 0
\end{array}
\right)$ if $\nu|NS$.

\noindent By strong approximation, we can write
\begin{equation}\label{eq:decomp} \omega_{NS}=\delta g_\infty u\end{equation}
 with $\delta\in B^\s$, $g_\infty\in GL_2^+(\RR)$ and $u\in V_0(NS)$. We define a map
\begin{eqnarray}
\theta_S: H^1(\Phi_0(NS),\OO(\tp))&\to& H^1(\Phi_0(NS),\OO(\tp^{-1}))\nonumber
\end{eqnarray} as follows: let $\xi$ be a cocycle representing  the cohomology class $x$ in  $H^1(\Phi_0(NS),\OO(\tp))$, then $\theta_S(\xi)$ is represented by the cocycle $\xi'(g)=\ws(u)\xi(\delta g \delta^{-1})$.

\noindent Observe that for all $\gamma\in\Phi_0(NS)$ the character satisfies: $\tp(\delta\gamma\delta^{-1})=\tp(\gamma)^{-1}$. It is immediate to verify that the map $\theta_S$ is a well defined isomorphism of abelian groups.

\noindent Let $g_p\in B^\times$ be such that the double coset $\Phi_0(NS)g_p \Phi_0(NS)$ defines the Hecke operator $T_p$ (see for example \cite{Terracini03}).  Let us define on  $H^1(\Phi_0(NS),\OO(\tp^{-1}))$, $\widetilde T_p$ as  the operator defined by the double coset $$\Phi_0(NS)\delta^{-1} g_p\delta\Phi_0(NS).$$ Notice that for a different decomposition  $\omega_{NS}=\delta' g'_\infty u'$ we have $\delta^{-1}\delta'\in\Phi_0(NS)$, so that there is an equality of double cosets $\Phi_0(NS)\delta^{-1} g_p\delta \Phi_0(NS)=\Phi_0(NS)\delta'^{-1} g_p\delta'\Phi_0(NS)$.

\begin{Lem1}\label{lemmad}
The following equalities hold:
\begin{itemize}
\item[a)] $\theta_S(T_p x)=\widetilde T_p(\theta_S(x))$ for all $p$ and $x\in H^1(\Phi_0(NS),\OO(\tp))$;
\item[b)] $b^\s_S\s(T_p x,y)=b^\s_S(x,\widetilde T_p y)$ for all $x\in H^1(\Phi_0(NS),\OO(\tp)), y\in H^1(\Phi_0(NS),\OO(\tp^{-1})$.
\item[c)] If $p$ does not divide $NS$ then $\widetilde T_p =\psi(p) T_p $\end{itemize}
\end{Lem1}

\proof
$a)$  Decompose $\Phi_0(NS)g_p\Phi_0(NS)= \coprod_i\Phi_0(NS)h_i$, with $h_i\in B^\s$.
Let $\xi$ be a cocycle representing the cohomology class $x$ in $H_1(\Phi_0(NS),\OO(\tp))$.\\
For $\gamma\in\Phi_0(NS)$, write \begin{eqnarray}\label{hi} h_i\delta\gamma\delta^{-1}=\gamma_i h_{j(i)};\end{eqnarray} so:
\begin{eqnarray}
\theta_S(T_p(\xi))(\gamma)&\stackrel{\textnormal{Def. of}\ \theta_S}{=}&\widehat\psi(u)(T_p(\xi))(\delta\gamma\delta^{-1})\nonumber\\
&\stackrel{\textnormal{Def. of}\ T_p}{=}&\widehat\psi(u)\sum_i\psi(h_i)\xi(\gamma_i)\nonumber\\
\end{eqnarray}

\noindent where $\psi(h_i)$ is defined as follows: first we observe that the elements $h_i\in V_0(NS)_q$  for all place $q$, where $V_0(NS)_q$ is the local component at $q$ of $V_0(NS)$. Then $$\psi(h_i)=\prod_{\begin{array}{l}q|N\\ q\not|p\ell \end{array}}\psi_q(h_i^{-1})\s\chi(h_i^{-1}). $$
Since $\delta^{-1} \Phi_0(NS)\delta =\Phi_0(NS)$, we have that $\Phi_0(NS)\delta=\delta \Phi_0(NS)$. Then: \begin{eqnarray}\label{dec}
\Phi_0(NS)\delta^{-1}g_p\delta\Phi_0(NS)&=&\delta^{-1}\Phi_0(NS)g_p\Phi_0(NS)\delta\nonumber\\
&=&\coprod_i\delta^{-1}\Phi_0(NS)h_i\delta\nonumber\\
&=&\coprod_i\Phi_0(NS)\delta^{-1}h_i\delta\label{coset}.
\end{eqnarray}
For $\gamma\in\Phi_0(NS)$ write $$\delta^{-1}h_i\delta\gamma\stackrel{by (\ref{hi})}{=}\delta^{-1}h_i\delta\gamma\delta^{-1}\delta=\delta^{-1}\gamma_ih_{j(i)}\delta=\delta^{-1}\gamma_i\delta\ \delta^{-1}h_{j(i)}\delta$$
\begin{eqnarray}
\widetilde T_p(\theta_S(\xi))(\gamma)&=&\sum_i\psi^{-1}(\delta^{-1}h_i\delta)\theta_S(\xi)(\delta^{-1}\gamma_i\delta)\nonumber\\
&=&\sum_i\psi(h_i)\widehat\psi(u)\xi(\delta\delta^{-1}\gamma_i\delta\delta^{-1})\nonumber\\
&=&\sum_i\psi(h_i)\widehat\psi(u)\xi(\gamma_i)\nonumber \end{eqnarray}

$b)$ Notice that the double coset defining the operator $\tilde T_p$ coincides with $\Phi_0(NS)g_p^\iota\Phi_0(NS)$ where $g_p^\iota=\nu(g_p)g_p^{-1}$. It is well known (see for example \cite[Chapter 6]{Hida93b}) that the operator associated to the double coset $\Phi_0(NS)g_p^\iota\Phi_0(NS)$ is the adjoint of the operator associated to $\Phi_0(NS)g_p\Phi_0(NS)$ under the cup product.\par

\noindent $c)$ See for example \cite{Shimura71} Section 3.4.5.\qed

For $x,y\in H^1(\XX_0(N,S),\OO(\ws))$ we define
$$\langle x,y\rangle_S=b^\s_S(x,\theta_S(y))$$
\begin{theorem}\label{teo.dual}
The pairing $\langle .,.\rangle_S$ on $H^1(\XX_0(N,S),\OO(\ws))$ satisfies the following properties:
\begin{enumerate}
\item it is perfect;
\item it is  Hecke-equivariant.
%\item it is alternating.
\end{enumerate}\end{theorem}

\proof\\
$1.$ By proposition (\ref{perfect}) the pairing $\langle .,.\rangle_S$ is perfect.\\
$2.$ It is Hecke-equivariant, infact for all $x,y\in  H^1(\Phi_0(NS),\OO(\tp))$:
\begin{eqnarray}
\langle T_px,y\rangle_S&=& b^\s_S(T_p x,\theta_S(y))\stackrel{Lemma \ref{lemmad}, b)}{=} b^\s_S(x,\widetilde T_p(\theta_S(y)))\nonumber\\
&\stackrel{Lemma \ref{lemmad}, a)}{=}& b^\s_S(x,\theta_S(T_p y))=\langle x,T_p y\rangle_S\nonumber
\end{eqnarray}
%$3.$ Since it comes from the cup product, the pairing $\langle .,.\rangle_S$ is alternating, that is for all  $x,y\in  H^1(\Phi_0(NS),\OO(\tp))$ $\langle x,y\rangle_S=-\langle y,x\rangle_S$.
\qed

A straightforward calculation allows to establish the following facts:
 \begin{Propo1}\label{prop:scambio}
Let $S'=S\cup\{q\}$ where $q$ is a prime number not dividing $\Delta NS\ell$. Then for all $x\in H^1(\Phi_0(N,S),\OO(\tp))$
\begin{itemize}
\item[a)] $\theta_{S'}(\Res_{\Phi_0(NS)/\Phi_0(NS')}(x))=(\theta_{S}(x))|_{\eta_q}$.
\item[b)] $\theta_{S'}(x|_{\eta_q})=\Res_{\Phi_0(NS)/\Phi_0(NS')}\theta_{S}(x)$.
\end{itemize}
\end{Propo1}

By recalling that restriction and corestriction are mutually adjoint with respect to cup product, that $b^\s_{S'}(x|_{\eta_q},y|_{\eta_q})=(q+1)b^\s_S(x,y)$ and that $T_q(x)= \Cor(x|_{\eta_q})$ (see for example \cite{Shimura71} we obtain the following formulas:
\begin{Lem1}\label{lem:dualita}
Let $S'=S\cup\{q\}$ where $q$ is a prime number not dividing $\Delta NS\ell$. The following properties hold for all $x,y\in H^1(\Phi_0(NS),\OO(\tp))$:
\begin{itemize}
\item[a)] $\langle\Res(x),\Res(y)\rangle_{S'}=\langle \psi^{-1}(q)T_q x,y\rangle_S$;
\item[b)] $\langle\Res(x),y|_{\eta_q}\rangle_{S'}=\langle x|_{\eta_q},\Res(y)\rangle_{S'}=(q+1)\langle x,y\rangle_S$.
\item[c)] $\langle x|_{\eta_q},y|_{\eta_q}\rangle_{S'}=\langle T_q x,y\rangle_S$.
\end{itemize}
\end{Lem1}

\subsection{Proof of Theorem \ref{conj2}}

 The proof is by induction on $|S|$. If $S$ is empty, the result has been proven in \cite{miri2006, Terracini03} by the construction of a suitable Taylor-Wiles system. Then we assume that the claim is true for $S$ and we deduce if for $S'=S\cup\{q\}$.
Suppose firstly that the polynomial
\begin{equation}\label{eq:polinomio} X^2-T_qX+q\psi(q) \end{equation}
has distinct roots $a,b\in k$ modulo $\mm_S$. Suppose moreover that $q^2\not \equiv 1 \hbox{ mod } \ell$. Then we can suppose that $a\not = qb$ in $k$. Since $\TT_{S}$ is an Henselian ring, there exists a root $A\in \TT_{S}$  such that $A$ reduces to $a$ modulo $\mm_{S}$. Define a map $\tilde\alpha: \MM_S\to \MM_{S'}$ by $\tilde\alpha(x)={\Res}(Ax) -x|_{\eta_q}$. Then $\tilde\alpha$ is a map of $\TT_{S'}$ modules and by Conjecture \ref{noscon} it is injective with torsion-free cokernel. Let $\tilde\beta$ be the transpose of $\tilde\alpha$ w.r.t. the pairings $\langle\cdot,\cdot\rangle_S$ and $\langle\cdot,\cdot\rangle_{S'}$; by Lemma \ref{lem:dualita} we compute

$$\langle\tilde\alpha(x),\tilde\alpha(y)\rangle_{S'} = \langle \psi^{-1}(q)T_qA^2x,y\rangle_{S}-2(q+1)\langle Ax,y\rangle_S+\langle T_q x,y\rangle_S$$
so that (notice that $A$ is a unit in $\TT_S$) the composition $\tilde\beta\circ\tilde\alpha$ is the multiplication by
\begin{eqnarray*} c&=&\psi^{-1}(q)T_qA^2-2(q+1)A+T_q\\
&=& \frac{\psi^{-1}(q)}{A}(A^2-q\psi(q))(A^2-\psi(q)).\end{eqnarray*}
On the other hand, the form $f$ determines maps $\pi_{S'}:\TT_{S'}\to\OO$ and $\pi_S:\TT_S\to\OO$ with kernels $\gp_{S'}$ and $\gp_S$ respectively. We want to compute
$$K=\ker(\gp_{S'}/\gp_{S'}^2\to \gp_S/\gp_S^2).$$
 It is easy to show with the techniques employed in \cite{miri2006, Terracini03} that the versal local deformation ring of $\bar\rho_q$ without ramification conditions concides with the versal ring with a ramification conditions unless the congruence
\begin{equation}\label{eq:congruencecond} a_q(f)^2\equiv \psi(q)(q+1)^2\ \modulo \ \lambda\end{equation}
 is satisfied. Therefore, $K$ is zero if (\ref{eq:congruencecond}) is not satified. Suppose now that (\ref{eq:congruencecond}) is satisfied. In this case the versal ring without ramification conditions is isomorphic to $\OO[[X,Y]]/(XY)$, where $X$ is equal, up to units, to  $Trace(\rho(Frob_q))^2-\psi(q)(q+1)^2$; by imposing the non-ramification condition $Y$ annihilates and the corresponding versal  becomes $\OO[[X]]$. By the methods of \cite{miri2006, Terracini03} it is easy to see that $K$ is the $\OO$-module generated by the image of $Y$ in $\gp_{S'}/\gp_{S'}^2$ and that it is isomorphic to $\OO/(a_q(f)^2-\psi(q)(q+1)^2)$. Notice that
$$T_q^2-\psi(q)(q+1)^2=A^{-2}(A^2-q^2\psi(q))(A^2-\psi(q))$$
and under our hypotheses $v_{\lambda}(c)=v_{\lambda}(\pi_{S'}(T_q^2-\psi(q)(q+1)^2)$.
Thus the claim is proved if $a\not=b$.
Assume now that $a=b$ and consider the Ihara map
\begin{eqnarray*} \tilde\alpha:{\mathcal M}_S^2\to {\mathcal M}_{S'}\end{eqnarray*}
Since $q^2\not\equiv 1 \hbox{ mod } \ell$, $\mathcal{B}_{S'}=\mathcal{B}_S$ so that $\TT_{S'}=\TT_S$ and $\tilde\alpha$ is an isomorphism. Therefore   ${\mathcal M}_{S'}$ is free of rank $4$ over $\TT_{S'}$. We have $\Omega_{S'}\simeq \Omega_S^2$ so that $lenght_\OO(\Omega_{S'})=2lenght_\OO(\Omega_S)=4lenght_\OO(\gp/{\gp^2})$ so that the map $\mR_{S'}\to \TT_{S'}=\TT_S$ is an isomorphism of complete intersection.\par

We remark that in the case $a=b$ the full Hecke algebra (with the $T_q$-operator) $\TT_{S',cohom}$ acting on $\MM_S'$ is isomorphic to $\TT_S[X]/(X^2-T_qX+q\psi(q))$ so that ${\mathcal M}_{S'}$ is free of rank $2$ over $\TT_{S',cohom}$.

\subsection{A consequence about congruences of modular forms}

\noindent   Under Conjecture \ref{noscon}, we can prove the following result about raising the level of modular forms:
\begin{theorem}\label{teo}
Assume Conjecture \ref{noscon}.
Let $f=\sum a_nq^n$ be a normalized newform in $S_2(\Gamma_0(N\Delta'\ell^2),\psi)$, supercuspidal of type $\tau=\chi\oplus\chi^\sigma$ at $\ell$, special at every prime $p|\Delta'$. Let $q$ be a prime not dividing $N\Delta$ such that $q^2\not\equiv 1\ \modulo \ell$; then there exists $g\in S_2(\Gamma_0(qN\Delta'\ell^2),\psi)$ supercuspidal of type $\tau$ at $\ell$, special at every prime $p|\Delta'$ such that $f\equiv g\ \modulo\ \lambda$ if and only if \begin{equation}\label{eq:condizione} a_q^2\equiv\psi(q)(q+1)^2\ \modulo\ \lambda\end{equation}.
\end{theorem}

\begin{proof} By Theorem \ref{conj2} and its proof we know that Condition (\ref{eq:condizione}) is equivalent to the non injectivity of the map $T_{S'}\to T_S$.
\end{proof}

\noindent If we do not make the  additional hypotheses on the local type the modular forms, then the above Conjecture is the classical \lq\lq raising the level\rq\rq problem that was first addressed in \cite{Ribet83b}.\\

{}

\bibliographystyle{alpha}
\bibliography{biblio}

\bigskip

\begin{center}
Miriam Ciavarella\\
Dipartimento di Matematica\\
Universit\`a degli Studi di Torino\\
Via Carlo Alberto 8, 10123 Torino, Italy\\

\bigskip

Lea Terracini\\
Dipartimento di Matematica\\
Universit\`a degli Studi di Torino\\
Via Carlo Alberto 8, 10123 Torino, Italy
\end{center}

\end{document}